\documentclass[12pt, a4paper]{article}
\usepackage[english]{babel}
\usepackage{amsmath,amssymb,amsthm}
\usepackage{fullpage}
\usepackage{textcomp}

\theoremstyle{plain}
\newtheorem{theorem}{Theorem}[section]

\newtheorem*{theorem1}{Theorem}

\theoremstyle{definition}

\theoremstyle{remark}

\begin{document}



\title{On the uniqueness of a solution to a stationary convection-diffusion equation with a generalized divergence-free drift}
\author{M.D. Surnachev\thanks{Email: peitsche@yandex.ru}\\ {\small Computational Aeroacoustcs Laboratory, Keldysh Institute of Applied Mathematics RAS} \\ {\small Miusskaya Sq. 4, Moscow 125047, Russia}}
\date{}
\maketitle
\vspace{-24pt}
\begin{abstract}
Let $A$ be a skew-symmetric matrix in $L^2(\Omega)$, $\Omega$ --- a bounded Lipschitz domain in $\mathbb{R}^n$, $n\geq 2$. The Dirichlet problem $ -\mathrm{div}\, (\nabla u+A\nabla u)=f$, $u\in H_0^1(\Omega)$, $f\in W^{-1,2}(\Omega)$ has at least one solution obtained by approximating $A$ and passing to the limit. In 2004 V.V. Zhikov  constructed an example of nonuniqueness. In the same paper he proved the uniqueness of solutions if the $L^p(\Omega)$ norms of $A$ are $o(p)$ as $p$ goes to infinity. We prove the uniqueness of solutions if $\exp (\gamma |A|)\in L^1(\Omega)$ for some $\gamma>0$, which generalizes Zhikov's theorem.
\end{abstract}



\begin{flushright}
{\it Dedicated to the memory of Academician V.I. Smirnov,}\\
{\it One of the Founding Fathers of MathPhys in Russia}
\end{flushright}

\section{Introduction}

Let $\Omega$ be a bounded Lipschitz domain in $\mathbb{R}^n$, $n\geq 2$, $f$ an element of  $W^{-1,2}(\Omega)$ and $A$ a skew-symmetric matrix from $L^2(\Omega)$. In this paper we are concerned with the question of uniqueness of solutions to the Dirichlet problem
\begin{equation}\label{Eq1}
Lu=-\mathrm{div}\, \left(\nabla u +A\nabla u\right)=f, \quad u\in W_0^{1,2}(\Omega).
\end{equation}
By a solution we mean a function $u\in W_0^{1,2}(\Omega)$ such that the integral identity
\begin{equation}\label{IntId}
\int_\Omega \left(\nabla u+A\nabla u \right)\nabla \varphi\, dx=(f,\varphi)
\end{equation}
holds for any $\varphi\in C_0^\infty(\Omega)$. 

Let us elucidate the term ``generalized drift'' in the paper title. Formally,
$$
(A_{ij}u_{x_j})_{x_i}=A_{ij}u_{x_i x_j}+u_{x_j} A_{ij,x_i}= (uA_{ij,x_i})_{x_j}-u A_{ij, x_i x_j}=(uA_{ij,x_i})_{x_j}.
$$
Here and below we use the Einstein convention of summation over repeated indices. For scalar and vector functional spaces we use the same notation, i.e. for $A:\Omega\to \mathbb{R}^{n\times n}$ we write $A\in L^2(\Omega)$ instead of $A \in L^2(\Omega)^{n\times n}$, for a vector field $a:\Omega\to \mathbb{R}^n$ we write $a\in L^2(\Omega)$ instead of $a\in L^2(\Omega)^n$ etc. 

More rigorously, if $A \in W^{1,2}(\Omega)$  and $\varphi \in C_0^\infty(\Omega)$,
\begin{gather*} 
\langle-\mathrm{div}\,(A\nabla u), \varphi \rangle=\int_\Omega A_{ij}u_{x_j}\varphi_{x_i}\, dx=-\int_\Omega A_{ij,x_j}u\varphi_{x_i}\, dx-\int_\Omega uA_{ij}\varphi_{x_ix_j}\, dx\\
=\int_\Omega A_{ji,x_j}u\varphi_{x_i}\, dx= \int_\Omega (u\,\mathrm{div}\, A) \nabla \varphi\, dx= \langle -\mathrm{div}\, (u\,\mathrm{div}\,A), \varphi \rangle,
\end{gather*}
where $(\mathrm{div} A)_i =A_{ji,x_j}$. Since $A$ is skew-symmetric, $\mathrm{div}\, \mathrm{div}\, A=A_{ij,x_ix_j}=0$. Thus, for a skew-symmetric $A\in W_0^{1,2}(\Omega)$ the Dirichlet problem \eqref{Eq1} can be written in the form 
\begin{equation}\label{Eq2}
-\mathrm{div}\, (\nabla u+a\nabla u)=f, \quad u\in W_0^{1,2}(\Omega), \quad f\in W^{-1,2}(\Omega),
\end{equation}
with the solenoidal vector field $a=\mathrm{div}\, A$ ($a_i= A_{ji,x_j}$). For nonsmooth $A$ one can say \cite{FP}, \cite{Z97}  that \eqref{Eq1} describes ``diffusion in a turbulent flow'' (in our case, stationary) since the flow velocity $a=\mathrm{div}\, A$ exists only in the sense of distributions. A similar class of equations in ``generalized divergence form'' was studied in  \cite{Osada}.

On the other hand, given a (smooth) solenoidal vector field $a$ we can construct (at least, locally) a skew-symmetric matrix $A$ such that $a=\mathrm{div}\, A$. Indeed, solenoidal $a$ corresponds to the closed $(n-1)$ form $\omega=\ast (a_i dx^i)$ ($\ast$ --- the Hodge star operator). By the Poincar\'e lemma (for instance, \cite{Spivak}) it is also exact, $\omega=d\alpha$ for $(n-2)$ form $\alpha$, provided that $\Omega$ is star-shaped (or contractible to a point, or diffeomorphic to a ball). The coefficients of $\alpha$ give the coefficients of $A$. In the language of differential forms, the passage between \eqref{Eq1} and \eqref{Eq2} is equivalent to the relation $\int_{\partial D} u d\alpha =(-1)^{n-1}\int_{\partial D} \alpha \wedge du$, $D$ a subdomain of $\Omega$, which follows from $d( ud\alpha+(-1)^n \alpha\wedge du)=0$. 

The form $\alpha$ can be additionally normed by $\delta\alpha=0$ ($\delta$ --- codifferential), and sought in the form $\alpha=\delta \beta$, which eventually leads to the problem $(d\delta +\delta d)\beta=\omega$ with suitable boundary conditions. For the problem $a=\mathrm{div}\,A$ the condition $\delta\alpha=0$ is equivalent to $A_{ij,x_k}+A_{jk,x_i}+A_{ki,x_j}=0$, which is neccesary for the representation of $A$ in the form of the rotor of a vector field $V$, i.e. $A_{ij}=\mathrm{curl}_{ij} V=V_{i,x_j}-V_{j,x_i}$. More on the Hodge decomposition for differential forms can be found in the famous Morrey's monography \cite[Chapter 7]{Morrey}. A rather complete theory of differential forms on Lipschitz domain was constructed in \cite{MMS} in the framework of Besov spaces.  

 In dimension $2$ this reduces to
$$
A=\begin{pmatrix}
0 & \alpha\\
-\alpha &0
\end{pmatrix}, \quad \mathrm{div}\, A= \left(-\alpha_y, \alpha_x\right)=(a_1,a_2). 
$$ 
Since $a$ is solenoidal the vector field $V=(a_2,-a_1)$ is potential. So one needs to find a function $\alpha$ with the given gradient $V$. In other words $a=\nabla^\perp \alpha$. 

In dimension $3$, any skew-symmetric matrix can be represented as $Ax=w\times x$,
and the problem of finding $A$ such that $a=\mathrm{div}\, A$ reads as $\mathrm{rot}\, w=a$, which is also easy to see from 
\begin{equation}\label{R1}
\mathrm{div}\, (w\times \nabla u)=\nabla u \cdot\mathrm{rot}\, w-w\cdot\mathrm{rot}\, \nabla u=\mathrm{div}\, (u\,\mathrm{rot}\, w).
\end{equation}
The problem of finding a vector field with prescribed rotor (and divergence) is a classical problem of vector calculus. For $\Omega=\{|x|<R\}$ one of solutions obtained by the Poincar\'e lemma is $w(x)=\int_0^1 a(tx)\times tx\, dt$. 

If $\Omega=\mathbb{R}^n$ and the solenoidal vector field $a$ is vanishing at infinity, a solution to $\mathrm{div}\, A=a$ can be obtained as the curl of the newtonian potential of $a$:
\begin{equation}\label{FR1}
A_{ij}=V_{i,x_j}-V_{j,x_i}, \quad V_i(x)=(n(n-2)\omega_n)^{-1}\int_{\mathbb{R}^n} a_i(y) |x-y|^{2-n}\, dy,\quad n\geq 3,
\end{equation}
where $\omega_n$ is the volume of the unit ball in $\mathbb{R}^n$, with the obvious modification for $n=2$. If $\Omega$ is a bounded domain and the normal component of $a$ on the boundary of $\Omega$ is equal to zero, then a solution to $\mathrm{div}\, A=a$ is given by the same formula \eqref{FR1}, where $a$ is extended by zero outside $\Omega$ (this extension is also solenoidal). 

In dimension $3$ formula \eqref{FR1} represents the standard vector calculus solution to $\mathrm{rot}\, w=a$ defined as $w= \mathrm{rot}\, (4\pi)^{-1}\int_{\mathbb{R}^n} a(y) |x-y|^{-1}\, dy$, which follows from representing $w=\mathrm{rot}\, v$ and using the vector calculus identity $\mathrm{rot}\, \mathrm{rot}\, v=\nabla\mathrm{div}\,v-\triangle v$. Such representation is of course only possible under the condition $\mathrm{div}\, w=0$, which is equivalent to requiring $\delta \alpha=0$ above.

 If the normal component of $a$ on the boundary is not equal to zero, one can continue $a$ to a sufficiently large ball $B$ which contains $\Omega$ by solving the auxilliary Neumann problem $-\triangle u=0$ in $B\setminus \Omega$,  $\partial u/\partial n=a\cdot n$ on $\partial \Omega$, $\partial u/\partial n=0$ on $\partial B$ ($n$ --- the exterior unit normal to $B\setminus \Omega$). Then one sets $a=\nabla u$ in $B\setminus \Omega$, $a=0$ in $\mathbb{R}^n\setminus B$, and a solution to $\mathrm{div}\, A=a$ is given by \eqref{FR1}. This construction assumes that either $\Omega$ does not have holes, or the flow of $a$ across the boundary of each hole is zero. 

If $\Omega$ has holes, the representation $a=\mathrm{div}\, A$ is obviously not always possible, but by the Hodge (Weyl in 3D) theorem there exists a harmonic (irrotational solenoidal) vector field $b$ such that $a=b+\mathrm{div}\, A$. For instance, one can take $b=\sum_j c_j \nabla \Gamma(x-x_j)$ where $x_j$ is a point inside the $j$-th hole, $\Gamma$ is the fundamental solution of the Laplace, and the constants $c_j$ are chosen to balance the flux of $a$ across the boundary of the corresponding hole. In detail this construction is discussed in \cite{KMPT}.

Another way is to directly solve the problem $-\triangle V=a$ in $\Omega$, $V\times n=0$ and  $\mathrm{div}\, V =0$ on $\partial \Omega$, and find $A=\mathrm{curl}\, V$. Regarding the equation $\mathrm{rot}\, u=f$ and corresponding boundary value problems see \cite{Kress} (classical potential theory), \cite{MMP} (modern potential theory) and recent papers \cite{Dub2013,Dub2014} (Galerkin's method).  For the closely related problem of finding a solenoidal vector field with prescribed boundary value (or a vector field with given divergence) we refer the reader to \cite{L1,L2,L3,KP}.

\section{Approximation solutions}

It is easy to prove that \eqref{Eq1} has at least one solution. Indeed, take a sequence $A_N$ of bounded skew-symmeric matrices converging to $A$ in $L^2(\Omega)$. Let $u_N$ be solutions to the corresponding problems
$$
-\mathrm{div}\, \left(\nabla u_n +A_N\nabla u\right)=f, \quad u_N\in W_0^{1,2}(\Omega),
$$
i.e.
$$
\int_\Omega \left(\nabla u_N+A_N\nabla u_N \right)\nabla \varphi\, dx=(f,\varphi)\quad\text{for all} \quad \varphi \in C_0^\infty(\Omega).
$$
By the Lax-Milgram lemma such solutions exist and are uniqely defined.  Using the test-function $u_N$ in the corresponding integral identity, we have
\begin{equation}\label{IdN}
\int_\Omega |\nabla u_N|^2\, dx= (f,u_N),
\end{equation}
wherefrom
$$
\int_\Omega |\nabla u_N|^2\, dx\leq \int_\Omega f^2\, dx.
$$
Extracting from $u_N$ a weakly convergent in $W_0^{1,2}(\Omega)$ subsequence  and passing to the limit in the integral identity
$$
\int_\Omega  \left(\nabla u_N+A_N\nabla u_N\right)\nabla \varphi\, dx=(f,u_N),
$$
we obtain a solution $u$ to \eqref{Eq1}. Passing to the limit in \eqref{IdN} we see that this solution satisfies the energy inequality
\begin{equation}\label{IdIneq}
\int_\Omega |\nabla u|^2\, dx\leq (f,u).
\end{equation}
Following  Zhikov \cite{Zhikov2004} we call a solution constructed by this procedure an appoximation solution. In the same paper V.V. Zhikov constructed an example of nonapproximation solutions, which satisfy the ``unnatural'' energy inequality
\begin{equation}
\int_\Omega |\nabla u|^2\, dx > (f,u).
\end{equation}
Denote
$$
[u,\varphi]=\int_\Omega A\nabla u\nabla \varphi,
$$
so that \eqref{IntId} can be rewritten as
\begin{equation}\label{BrId}
\int_\Omega \nabla u \nabla \varphi\, dx+[u,\varphi]=(f,\varphi).
\end{equation}
It is clear that
\begin{equation}\label{B}
|[u,\varphi]|\leq C \|\nabla \varphi\|_{L^2(\Omega)} \quad \text{for all}\quad \varphi\in C_0^\infty(\Omega),
\end{equation}
and $[u,\varphi]$, initially defined for $\varphi\in C_0^\infty(\Omega)$, can be extended to a linear bounded functional on $W_0^{1,2}(\Omega)$. Accordingly, in \eqref{BrId} the set of admissible test functions can be extended to $W_0^{1,2}(\Omega)$. Substituting $u$ as a test function in \eqref{BrId} we obtain
\begin{equation}\label{IntId2}
\int_\Omega |\nabla u|^2\, dx+[u,u]=(f,u).
\end{equation}

On the other hand, any $u\in W_0^{1,2}(\Omega)$ satisfying \eqref{B}, is a solution to \eqref{Eq1} with the right-hand side $f$ defined by \eqref{BrId}. So, the set of functions $u\in W_0^{1,2}(\Omega)$ satisfying \eqref{B} is the set of all solutions to \eqref{Eq1} when $f$ ranges over $W^{-1,2}(\Omega)$. For a given skew-symmetric matrix $A$ we denote this set by $D(A)$. When necessary to distinguish between different matrices, we add a subscript to $[\cdot,\cdot]$: for instance, $[u,v]_A$.

The rest of this section is devoted to certain elementary observations. Inequality \eqref{IdIneq} translates into $[u,u]\geq 0$ for approximation solutions. The idea of Zhikov was to find an example of $[u,u]<0$. Since an approximation solution always exists this immediately implies nonuniqueness. On the other hand, if $[u,u]\geq 0$ for all $u\in W_0^{1,2}(\Omega)$ then for any right-hand side $f$ a solution is unique, and \eqref{IdIneq} holds.  Another easy observation is that $[u,u]=0$ for all $u\in W_0^{1,2}(\Omega)$ is equivalent to the uniqueness of solutions together with the energy identity 
\begin{equation}\label{EnId}
\int_\Omega |\nabla u|^2\, dx=(f,u)
\end{equation}
for all $f$.

Also note that if there exists $u$ with $[u,u]>0$, then for problem \eqref{Eq1} with $A$ replaced by $-A$ there exists a nonapproximation solution. Analogously, if for a given matrix $A$ there exists a solution with $[u,u]<0$ then for problem \eqref{Eq1} with $A$ replaced by $-A$ there exists a solution which satisfies the strict energy inequality 
$$
\int_\Omega |\nabla u|^2\, dx<(f,u).
$$
If for some right-hand side $f$ there exist multiple solutions, then there exists a nontrivial solution $u_0$ corresponding to $f=0$. For $u_0$ identity \eqref{IntId2} gives
$$
[u_0,u_0]=-\int_\Omega |\nabla u_0|^2\, dx<0.
$$
Since for any solution $u$ there holds
\begin{equation}\label{F}
[u+tu_0,u+t u_0]= [u,u] +t[u_0,u]+t[u,u_0]+t^2[u_0,u_0]<0\quad\text{for large}\quad t,
\end{equation}
then nonuniqueness for some $f$ implies nonuniqueness for all right-hand sides $f$. 

The same observation also allows us to single out an ``extremal'' solution from $L^{-1}f$. Indeed, consider $I[f]=\sup \{[u,u],\ u\in L^{-1}f\}$. Since an approximation solution always exists, $0\leq I[f]$. For solutions, satisying $[u,u]\geq 0$, $\|u\| \leq \|f\|$ and $[u,u]\leq (f,u)\leq \|f\|^2$. It follows that $I[f]\leq \|f\|^2$. Take a sequence $u_k\in L^{-1}f$ such that $[u_k,u_k]$ monotonically increases and converges to $I[f]$. Then one can easily verify that
$$
\frac{1}{2}\int_\Omega |\nabla (u_k -u_m)|^2\, dx=2\left[\frac{u_k+u_m}{2},\, \frac{u_k+u_m}{2} \right]- [u_k,u_k]-[u_m,u_m]\to 0
$$
as $k,m\to \infty$. Thus, $u_k\to u$ strongly in $W_0^{1,2}(\Omega)$, $Lu=f$ and 
$$
[u,u]=(f,u)-\int_\Omega |\nabla u|^2\, dx= \lim_{k\to \infty} (f,u_k)-\int_\Omega |\nabla u_k|^2\, dx= \lim_{k\to \infty} [u_k,u_k]=I[f].
$$
For any $z\in L^{-1}0$, $t\in\mathbb{R}$ we have $u+tz\in L^{-1} f$, so $[u+tz,u+tz]\leq [u,u]$. Hence $u$ satisfies
\begin{equation}\label{P}
[u,z]+[z,u]=0\quad\text{for any} \quad z\in L^{-1}0.
\end{equation}
From \eqref{F}, any function from $L^{-1}f$ satisfying the latter property maximizes $I[f]$ and is uniquely defined. Denote the special solution of $Lu=f$ which maximizes $I[f]$ by $\tilde{L}^{-1} f$. It is obvious that $\tilde{L}^{-1}f+\tilde{L}^{-1}g\in L^{-1} (f+g)$ and satisfies \eqref{P}. Therefore $\tilde{L}^{-1}f+\tilde{L}^{-1}g=\tilde{L}^{-1} (f+g)$. So, $\tilde{L}^{-1}f:W^{-1,2}(\Omega)\to W_0^{1,2}(\Omega)$ is a linear bounded operator, which is the right inverse for $L$.

For any skew-symmetric matrix $B\in L^\infty(\Omega)$ there holds 
$$
|[u,\varphi]|\leq C\|B\|_\infty \|\nabla u\|_{L^2(\Omega)}\|\nabla \varphi\|_{L^2(\Omega)},
$$
which implies
$$
 D(B)=W_0^{1,2}(\Omega),\quad [u,u]_B=\int_\Omega B\nabla u \nabla u\, dx=0.
$$
Thus, addition of any skew-symmetric matrix $B\in L^\infty(\Omega)$ to $A$  does not change $D(A)$ and $[u,u]$:
$$
D(A+B)=D(A), \quad [u,u]_{A+B}=[u,u]_A+[u,u]_B=[u,u]_A.
$$
In certain sense, the information on uniqueness/nonuniqueness is contained in the set of large values of $A$.
 In \cite{Zhikov2004} Zhikov proved the following
\begin{theorem1}[Zhikov]
Let 
\begin{equation}\label{ZC}
\lim_{p\to\infty} p^{-1} \|A\|_{L^p(\Omega)}=0.
\end{equation}
Then \eqref{Eq1} has a unique solution.
\end{theorem1}
The aim of this paper is to clarify and refine this result.

\section{Around BMO and $\mathcal{H}^1$}

Recall that $BMO$ is the set of locally integrable on $\mathbb{R}^n$ functions such that
$$
\|f\|_{BMO}=\sup \frac{1}{|Q|} \int_Q |f-f_Q|\, dx <\infty, \quad f_Q=\frac{1}{|Q|} \int_Q f\,dx,
$$
where the supremum is taken over all cubes $Q\subset \mathbb{R}^n$ with faces parallel to coordinate hyperplanes (or, alternatively, over all balls).

It is well known that $A\in BMO$ guarantees $D(A)=W_0^{1,2}(\Omega)$ and $[u,u]_A=0$. Indeed, for $u,v \in C_0^\infty(\Omega)$ write 
$$
\int_\Omega A_{ij} u_{x_j} v_{x_i}\, dx=\frac{1}{2} \int_\Omega A_{ij} (u_{x_j} v_{x_i} -u_{x_i} v_{x_j})\, dx.
$$
The crucial fact is that $u_{x_i} v_{x_j} -u_{x_j} v_{x_i}$ belongs to the Hardy space $\mathcal{H}^1(\mathbb{R}^n)$, and
$$
\|u_{x_i} v_{x_j}- u_{x_j} v_{x_i}\|_{\mathcal{H}^1(\mathbb{R}^n)} \leq C \|\nabla u\|_{L^2(\Omega)} \|\nabla v\|_{L^2(\Omega)}.
$$
This fact can be proved using the commutator theorem from \cite{CRW}. Much easier proof was given later in \cite{CLMS}. There is a number of different equivalent definitions of $\mathcal{H}^1(\mathbb{R}^n)$, the proof of \cite{CLMS} used the following one. Let $\Phi$ be a smooth compactly supported function with $\int \Phi\, dx=1$. Denote
$$
M_\Phi f(x)=\sup_{t>0} |\Phi_t\ast f|(x),\quad \Phi_t(x)=t^{-n} \Phi \left(\frac{x}{t} \right).
$$
 Then
$$
\mathcal{H}^1(\mathbb{R}^n)= \left\{f\in L^1(\mathbb{R}^n):\ M_\Phi f\in L^1(\mathbb{R}^n) \right\}
$$ 
Since $BMO$ is dual to $\mathcal{H}^1(\mathbb{R}^n)$ \cite{FeffermanStein},  we arrive at
$$
\int_\Omega A_{ij} u_{x_j} v_{x_i}\, dx\leq C \|A\|_{BMO} \|  \|\nabla u\|_{L^2(\Omega)} \|\nabla v\|_{L^2(\Omega)}, \quad C=C(n).
$$
Thus, the skew-symmetric bilinear form $[\cdot,\cdot]_A$ defined on $C_0^\infty(\Omega) \times C_0^{\infty}(\Omega)$ is continuos with respect to both arguments in the norm of $W_0^{1,2}(\Omega)$ and can be extended to the form on $W_0^{1,2}(\Omega)\times W_0^{1,2}(\Omega)$ satisfying 
$$
|[u,v]_A| \leq C  \|\nabla u\|_{L^2(\Omega)} \|\nabla v\|_{L^2(\Omega)}, \quad [u,v]_A=-[v,u]_A
$$
for all $u,v \in W_0^{1,2}(\Omega)$. Then the existence and uniqueness of a solution to \eqref{Eq1} follows from the Lax-Milgram lemma.

For other useful properties of $BMO$ and Hardy spaces we refer the reader to \cite{Stein93} (see also the  excellent expository article \cite{Semmes}).

A decade ago Maz'ya and Verbitsy \cite{MazVer} proved a reverse result. This result is formulated for a wide class of equations with lower-order terms. We cite here only the basic part which relates to \eqref{Eq1}. Let $L^{1,2}(\mathbb{R}^n)$ be the closure of smooth finite functions with respect to the norm $\|\nabla u\|_{L^2(\mathbb{R}^n)}$, and $L^{-1,2}(\mathbb{R}^n)$ be its dual. The operator
$$
-\mathrm{div}\, (A \nabla u):\, L^{1,2}(\mathbb{R}^n)\rightarrow L^{-1,2}(\mathbb{R}^n) 
$$
 is bounded if and only if 
\begin{gather*}
A^s=\frac{A+A^T}{2}\in L^\infty(\mathbb{R}^n),\quad \text{and}\\
\mathrm{div}\, A^c\in BMO^{-1}(\mathbb{R}^n)^n, \quad A^c=\frac{A-A^T}{2}.
\end{gather*}
Here $BMO^{-1}(\mathbb{R}^n)$ denotes the set of distributions which can be represented as the divergence of a $BMO$ vector field. So, there exists a matrix $\Phi$ with $BMO$ entries such that $\mathrm{div}\, A^c =\mathrm{div}\, \Phi$. In the sense of generalized functions, for $ u,v \in C_0^\infty(\mathbb{R}^n)$ we have
\begin{gather*}
\langle -\mathrm{div}\, (A \nabla u),\, v \rangle= \langle A^s \nabla u, \nabla v  \rangle - \langle \mathrm{div}\, A^c \nabla u, v \rangle\\
=  \langle A^s \nabla u, \nabla v  \rangle - \langle \mathrm{div}\, \Phi \nabla u, v \rangle= -\langle \mathrm{div}\, ((A^s+\Phi)\nabla u), v \rangle.
\end{gather*}
This means that on smooth finite functions the operator is identical to an analogous operator with symmetric part of the matrix bounded and skew-symmetric part from $BMO$. The skew-symmetric part $\Phi$ can be found from  $\Phi=-\triangle^{-1} \mathrm{curl}\, \mathrm{div}\, A^c$. Here the divergence operator acts on $a=a_{ij}$ as $\mathrm{div}_j\, a=\partial _{x_i} a_{ij}$, and the curl of $f=\{f_i\}$ is $\mathrm{curl}_{ij}f =\partial_{x_j} f_i-\partial_{x_i} f_j$. In dimension $2$ the matrix $A^c$ itself belongs to $BMO$. 

The functions from $BMO$ are exponentially summable (the John-Nirenberg lemma \cite{JN}), and satisfy
\begin{equation}\label{JN}
\frac{1}{|Q|} \int_Q |f-f_Q|^p \, dx \leq \left(C p \|f\|_{BMO} \right)^p, \quad C=C(n)
\end{equation}
for any $f\in BMO$ and cube $Q\subset \mathbb{R}^n$. Thus, for $A\in BMO$ the limit in \eqref{ZC} is always finite, but need not be zero, as can be demonstrated by the example of $\log |x|$. 

For $A\in BMO$ Zhikov proved the uniqueness of approximation solutions without using the $BMO$--$\mathcal{H}^1$ duality. In this case, it is sufficient to prove uniqueness for solutions corresponding to the set of bounded right-hand sides, which is dense in $W^{-1,2}(\Omega)$ (see \cite{Zhikov2004} for details). If $A\in BMO\cap L^\infty(\Omega)$, one can obtain the Meyers type estimate $$\|\nabla u\|_{L^q(\Omega)}\leq C \|f\|_{L^\infty(\Omega)}$$ 
for some $q>2$ and $C$ which depend only on $\|A\|_{BMO}$ and $\Omega$. Since $BMO$ functions are summable to any power, $A \nabla u \nabla u \in L^1(\Omega)$.  By H\"older's inequality
$$
|[u,\varphi]_A| \leq C \|\nabla u\| _{L^q(\Omega)} \|A\|_{L^r(\Omega)} \|\nabla \varphi\|_{L^2(\Omega)}, \quad r^{-1}=2^{-1}-q^{-1},
$$
for $\varphi \in C_0^\infty(\Omega)$. Approximating $u$ by such $\varphi$ we arrive at 
$$
[u,u]_A = \int_\Omega A \nabla u \nabla u=0,
$$
which implies uniqueness for approximation solutions corresponding to bounded right-hand sides. 

There is a variety results on equations of type \eqref{Eq1} with $A\in BMO$ (or equations of type \eqref{Eq2} with divergence-free $a\in BMO^{-1}$). See, for instance, the survey article \cite{FRV} on the magnetogeostrophic equation and \cite{FV,SSSZ} for results on regularity and qualitative theory of solutions.

\section{Main result}

Now we are ready to state the main result of this paper.
\begin{theorem}\label{T1}
Let the matrix $A$ satisfy the condition
\begin{equation}\label{ZC1}
\lim_{p\to\infty} p^{-1} \|A\|_{L^p(\Omega)}<\infty.
\end{equation}
Then \eqref{Eq1} has a unique solution.
\end{theorem}
By the John-Nirenberg estimate \eqref{JN}, matrices with $BMO$ elements satisfy \eqref{ZC1}. It is easy to see that \eqref{ZC1} is equivalent to the exponential summability of $A$:
\begin{equation}\label{exp}
\int_\Omega \exp (\gamma |A|)\, dx= \sum_{p=0}^{\infty} \frac{\gamma^p}{p!}\int_\Omega |A|^p \, dx,
\end{equation}
and by the Stirling formula
$$
 \frac{\gamma^p}{p!}\int_\Omega |A|^p \, dx\sim \frac{1}{\sqrt{2\pi p}} (Le\gamma )^p, \quad L=\lim_{p\to \infty}p^{-1} \|A\|_{L^p(\Omega)}.
$$
The series on the right-hand side of \eqref{exp} converge if $\gamma< (Le)^{-1}$. 

Let $M \exp (\gamma |A|)$ be the Hardy-Littlewood maximal function of  $\exp (\gamma |A|)\in L^1(\Omega)$, which is continued by zero outside $\Omega$. Clearly,
$$
|A| \leq \frac{1}{\gamma}\log M \exp (\gamma |A|).
$$ 
By the result of Coifman and Rochberg \cite{CR}, the right-hand side of the last expression is in $BMO$ with the $BMO$ ``norm'' bounded by $\gamma^{-1}C(n)$. So, \eqref{ZC1} is equivalent to $|A|$ having a $BMO$ majorant.

Let us note that the condition of exponential summability naturally arises in the theory of qusiharmonic vector fields with unbounded distortion \cite{IS}.

It is easy to give an example of function satisfying \eqref{ZC1} but not in $BMO$. It follows from the definition of $BMO$ that for two touching cubes of the same size there holds
\begin{equation}\label{BP}
|f_{Q_1}-f_{Q_2}| \leq 2^{n+1} \|f\|_{BMO}.
\end{equation}
Let $n=2$, $x=(x_1,x_2)$. Take $f(x)=\log |x|$ if $x_1x_2>0$ and $f=0$ otherwise. Clearly, for such function \eqref{BP} is not satisfied.

The condition of theorem \eqref{T1} is sufficient for the uniqueness but far from necessary. It is worth to note that the addition of a skew-symmetric matrix with zero divergence to matrix $A$  does not change the equation. Let $C\in L^2(\Omega)$ be a skew-symmetric matrix with $\mathrm{div}\, C=0$, $u\in W_0^{1,2}(\Omega)$ and $\varphi \in C_0^\infty(\Omega)$. We have
$$
\langle -\mathrm{div}\, (C\nabla u), \varphi\rangle =\int_\Omega C_{ij} u_{x_j} \varphi_{x_i}\, dx
=\int_\Omega C_{ij} (u\varphi)_{x_j}\, dx-\int_\Omega u C_{ij} \varphi_{x_ix_j}\, dx=0.
$$
In dimension $2$ this does not bring anything new since any skew-symmetric matrix $2\times2$ with zero divergence is of the form
$$
\begin{pmatrix}
0 & c \\
-c & 0
\end{pmatrix}, \quad c=\mathrm{const},
$$
and the addition of any bounded matrix to $A$ does not affect \eqref{ZC1}. In dimension $3$ the situation is more interesting. Write the skew-symmetric matrix $A$ as
$$
A =
\begin{pmatrix}
0 & -a_3 & a_2 \\
a_3 & 0 & -a_1 \\
-a_2 & a_1 & 0
\end{pmatrix}, \quad A\xi =a\times \xi, \quad a= (a_1,a_2,a_3).
$$
The condition of zero divergence leads to $\nabla \times a=0$, which is satisfied by $a=\nabla \varphi$. This can be also seen from \eqref{R1}. We can add any matrix of the form
$$
C(\varphi)=
\begin{pmatrix}
0 & -\varphi_{x_3} & \varphi_{x_2} \\
\varphi_{x_3} & 0 & -\varphi_{x_2} \\
-\varphi_{x_2} & \varphi_{x_1} & 0
\end{pmatrix},\quad \varphi \in W^{1,2}(\Omega), 
$$
to $A$ and the equation basically stays the same. This is the reason why in \cite{MazVer} the result is given in terms of equivalence classes for $n\geq 3$.

\section{Lipschitz truncations. The proof of the main result}
In this section we prove Theorem~\ref{T1}. The proof relies on the technique of Lipschitz trunctions. For the reader's convenience we briefly remind the details. Let $u\in W_0^{1,1}(\Omega)$ and $g=M\,|\nabla u|)$, where $M$ stands for the standard Hardy-Littlewood maximal function:
$$
Mf (x)=\sup \frac{1}{|B|} \int_B |f|\, dx,
$$ 
where the supremum is taken over all balls $B\subset \mathbb{R}^n$ which contain $x$ (uncentered maximal function)  or are centered at $x$ (centered maximal function). Then for almost all $x,y\in \Omega$ there holds
$$
|u(x)-u(y)|\leq C(n) |x-y| (g(x)+g(y)), \quad |u(x)|\leq C \mathrm{dist}\, (x,\partial \Omega) g(x).
$$
From these estimates it follows that on the set $F(\lambda)=\{g\leq \lambda\}\cup (\mathbb{R}^n\setminus \Omega)$ the function $u$ is Lipschitz with the Lipschitz constant $C\lambda$. Using the McShane theorem \cite{Mac}, we can extend $\left.u\right|_{F(\lambda)}$ to the whole space $\mathbb{R}^n$ with the same Lipschitz constant $C\lambda$. The resulting extension $u_\lambda$ is called the Lipschitz truncation of $u$. For further details on  Lipschitz truncations and their applications we recommend \cite{DR}.

 Let $u$ be a solution to \eqref{Eq1} with $f=0$, i.e.
 \begin{equation}\label{ZID}
 \int_\Omega (\nabla u+A\nabla u)\nabla \varphi\, dx=0 \quad\text{for all} \quad\varphi \in C_0^\infty(\Omega).
 \end{equation}
By approximation, one can take here Lipshitz $\varphi$ vanishing on $\partial \Omega$.
 
 Take the test function $\varphi=u_\lambda$ in \eqref{IntId}. Using the skew-symmetry of $A$ we obtain
 $$
 \int_{\{g\leq \lambda\}} |\nabla u|^2\, dx=-\int_{\{g>\lambda\}} (A+I)\nabla u\nabla u_\lambda\, dx \leq C \lambda \int_{\{g>\lambda\}} (|A|+1)|\nabla u|\, dx.
 $$
 Next, multiply this inequality by $\varepsilon \lambda^{-1-\varepsilon}$, $\varepsilon>0$, and integrate with respect to $\lambda$ from $1$ to $\infty$. Fubini's theorem yields
 $$
 \int_\Omega |\nabla u|^2 (\max (1,g))^{-\varepsilon}\, dx \leq \frac{C\varepsilon}{1-\varepsilon} \int_\Omega (|A|+1) |\nabla u| (g^{1-\varepsilon}-1)_{+}\, dx.
 $$
Using H\"older's inequality and the boundedness of the maximal function in $L^2$, for small $\varepsilon$ we obtain 
\begin{gather*}
 \int_\Omega |\nabla u|^2 (\max (1,g))^{-\varepsilon}\, dx\\
 \leq C\varepsilon\left(\int_\Omega|\nabla u|^2\, dx \right)^{1/2} \left(\int_\Omega g^2\, dx \right)^{(1-\varepsilon)/2} \left(\int_\Omega (|A|+1)^{2/\varepsilon} \right)^{\varepsilon/2}\\
 \leq C\varepsilon\left(\int_\Omega (|A|+1)^{2/\varepsilon}\,dx \right)^{\varepsilon/2}  \left(\int_\Omega |\nabla u|^2\, dx \right)^{1-\varepsilon/2}. 
\end{gather*}
 Passing to the limit as  $\varepsilon\to 0$ we arrive at 
 $$
\int_\Omega |\nabla u|^2\, dx \leq C \lim_{p\to \infty} \frac{\|A\|_{L^p(\Omega)}}{p}\int_\Omega |\nabla u|^2\, dx,\quad C=C(n).
 $$
 Therefore, $u=0$ provided that the limit in \eqref{ZC1} is small enough.  The theorem is thus proved for $A$ such that 
 \begin{equation}\label{ZC1a}
 \lim_{p\to\infty} p^{-1}\|A\|_{L^p(\Omega)}< C(n)
 \end{equation}
  for some positive constant $C(n)$. Let $A$ be a skew-symmetric matrix satisfying \eqref{ZC1}. Consider \eqref{Eq1} with $A$ replaced by $\pm tA$ with $t>0$ such that $tA$ satisfies \eqref{ZC1a}. Clearly, $D(\pm tA)=D(A)$ and $[u,u]_{\pm tA}=\pm t[u,u]_A$. For $tA$ we have uniqueness, so $[u,u]_{tA}\geq 0$ for all $u\in D(A)$. Similarly, $[u,u]_{-tA}\geq 0$ for all $u\in D(A)$. Thus, $[u,u]=0$ for all $u\in D(A)$. This immediately implies the uniqueness of solutions and validity of \eqref{EnId}. The proof of Theorem~\ref{T1} is complete. 
 
 \section{Corollaries. }
 
 In this section we focus on problem \eqref{Eq2} with ``standard'' solenoidal drift. A vector field $a\in L^1(\Omega)$ is called solenoidal (or divergence free) if $\mathrm{div}\, a=0$ in the sense of distributions, i.e.
 $$
 \int a\nabla \varphi\, dx=0 \quad\text{for all} \quad \varphi\in C_0^\infty(\Omega).
 $$
 A solution to \eqref{Eq2} is a function $u\in W_0^{1,2}(\Omega)$ which satisfies
\begin{equation}\label{IIvec}
 \int_\Omega (\nabla u+au)\nabla \varphi\, dx=(f,\varphi) \quad\text{for all} \quad \varphi\in C_0^\infty(\Omega).
 \end{equation}
 Using the same reasoning as above, one can show the existence of approximation solutions if the solenoidal vector field 
 \begin{equation}\label{AC}
  a\in L^{2n/(n+2)}(\Omega)\quad \text{for}\quad n\geq 3,\quad \text{and}\quad a\in L\log^{1/2}L(\Omega)\quad \text{for}\quad n=2.
 \end{equation}
 In view of the embedding theorem this condition guarantees $au\in L^1(\Omega)$. Denote 
 \begin{gather*}
 [u,\varphi]_a=\int_\Omega au\nabla \varphi\, dx,\quad u\in W_0^{1,2}(\Omega), \quad \varphi\in C_0^\infty(\Omega),\\
  D(a)= \{u\in W_0^{1,2}(\Omega):\ |[u,\varphi]_a| \leq C \|\nabla \varphi\|_{L^2(\Omega)}\quad \text{for all} \quad\varphi \in C_0^\infty(\Omega)\}.
 \end{gather*}
 As above, the set ${D}(a)$ coincides with the set of all solutions  to \eqref{Eq2}, for a solution $u$ the form $[u,\varphi]_a$ is extended to $\varphi \in W_0^{1,2}(\Omega)$, \eqref{IIvec} can be written in the form
 \eqref{BrId}, substituting $u$ as a test-function one obtains \eqref{IntId2}. Further on, if there is no ambiguity, we drop the subscript $a$ in the form $[u,\varphi]$.
 
 The simplest condition (apart from the trivial $a\in L^\infty(\Omega)$) which guarantees the existence and uniqueness of a solution is $a\in L^n(\Omega)$. For $n>2$, by the Sobolev embedding theorem,
\begin{equation}\label{LN}
\left|\int_\Omega au\nabla \varphi\, dx\right|\leq C \|a\|_{L^n(\Omega)} \|u\|_{L^2(\Omega)} \|\nabla \varphi\|_{L^2(\Omega)},
 \end{equation}
 so the form $[u,\varphi]$ is continuous with respect to both arguments in the norm of $W_0^{1,2}(\Omega)\times W_0^{1,2}(\Omega)$, and the existence and uniqueness of a solution follows from the Lax-Milgram lemma. 
 
 For $n=2$, let $\Omega$ be a simply-connected domain. Since $a=(a_1,a_2)$ is solenoidal, we can find $Q\in W^{1,2}(\Omega)$ such that $a_1=-Q_y$, $a_2=Q_x$. Rewrite \eqref{Eq2} in the form \eqref{Eq1}:
$$
 \int_\Omega u(a_1\varphi_x+a_2\varphi_y)\, dxdy=\int_\Omega Q (u_y \varphi_x -u_x\varphi_y)\, dxdy
$$
 for  all $u\in W_0^{1,2}(\Omega)$ and  $\varphi \in C_0^\infty(\Omega)$. Extend the function $Q$ to the whole plane so that $\|Q\|_{W^{1,2}(\mathbb{R}^n)}\leq \|Q\|_{W^{1,2}(\Omega)}$. By the Poincar\'e inequality, for any ball $B\subset \mathbb{R}^2$ there holds
 $$
 |B|^{-1}\int_{B} |Q-\bar Q|^2\, dx\leq \int_{B} |a|^2\, dx.
 $$
 Hence $Q\in BMO$ and $\|Q\|_{BMO}\leq C\|a\|_{L^2(\Omega)}$. Using the duality of $BMO$ and $\mathcal{H}^1$, we obtain \eqref{LN} for $n=2$.

A thorough study of regularity properties (boundedness, strong maximum principle, continuity, Harnack's inequality) of solutions of second-order linear elliptic and parabolic equations with ``rough'' divergence free drifts from $L^n$ and Morrey spaces generalizing $L^n$ was done by Nazarov and Uraltseva in \cite{NU}. Interesting examples are due to Filonov \cite{Filonov}.

 It is not hard to prove \cite{Z97} that $a\in L^2(\Omega)$ guarantees the uniqueness of solutions and validity of the energy identity \eqref{EnId}. Indeed, approximating $u$ by smooth functions one can prove that for $a\in L^2(\Omega)$ there holds
 $$
 \int_\Omega au\nabla \varphi\, dx=-\int_\Omega a\varphi \nabla u\, dx, \quad u\in W_0^{1,2}(\Omega),\  \varphi \in C_0^\infty(\Omega),
 $$
 so \eqref{IIvec} acquires the form 
 \begin{equation}\label{IIvec1}
  \int_\Omega \nabla u\nabla \varphi\, dx= \int_\Omega a\varphi \nabla u\, dx +(f,\varphi)\quad\text{for all} \quad \varphi\in C_0^\infty(\Omega).
 \end{equation}
 Approximating $T_k(u) = \max(\min (u,k), -k)$, $k>0$, by bounded smooth functions, we can set $\varphi=T_k(u)$ in \eqref{IIvec1}, which gives
 $$
 \int_\Omega |\nabla T_k(u)|^2\, dx-(f,T_k(u))=\int_\Omega a T_k(u)\nabla u\, dx=\int_\Omega a\nabla\left( \int_0^u T_k(s)\, ds\right) \, dx=0.
 $$
 Sending $k$ to infinity, and using $\nabla T_k(u)=\chi_{\{|u|<k\}}\nabla u$ a.e. in $\Omega$, we finally obtain energy identity \eqref{EnId}
 which implies the uniqueness.
 
 It was in fact the convection-diffusion equation in form \eqref{Eq2} for which Zhikov's example in \cite{Zhikov2004} was constructed. The example had the following form: $ \Omega=B_1\subset \mathbb{R}^3$, $a=a_0(x|x|^{-1}) x|x|^{-3}$, $u=(1-|x|^4)u_0(x|x|^{-1})$,  where $ \int_S a_0 \, d\sigma= \int_S u_0a_0\, d\sigma=0$ and $\int_S a_0u_0^2\, d\sigma=-2$, $S=\{|x|=1\}$. Using $u,\, a\nabla u\in L^\infty(\Omega)$ one can verify that $[u,u]=-\int_\Omega au\nabla u\, dx=-1$. A similar example for the problem $-\triangle u + b\nabla u + \mathrm{div}\, (bu)=f\in W^{-1,2}(\Omega)$, $u\in W_0^{1,2}(\Omega)$ was constructed  in \cite{BD2015}, where the question of existence and uniqueness was studied in the framework of renormalized solutions.
 
In \cite{Zhikov2004}  Zhikov proved the following result which improves the $L^2(\Omega)$ condition. 
\begin{theorem1}[Zhikov]
If the solenoidal vector field $a$ satisfies $\lim_{\varepsilon\to 0} \varepsilon\|a\|_{L^{2-\varepsilon}(\Omega)}=0$, then the approximation solution of \eqref{Eq2} is unique for each $f\in W^{-1,2}(\Omega)$.
\end{theorem1}
 
 In dimension $2$ this result can be strengthed.
 \begin{theorem}\label{T2}
 Let $n=2$ and the solenoidal vector field $a=(a_1,a_2)$ satisfy 
\begin{equation}\label{ZC2}
 \lim_{\varepsilon\to 0} \sqrt{\varepsilon} \|a\|_{L^{2-\varepsilon}(\Omega)}=0.
 \end{equation}
 Then for any $f\in W^{-1,2}(\Omega)$ equation \eqref{Eq2} has a unique solution.
 \end{theorem}
We shall obtain this theorem as a partial case of a more general statement.
 
Recall that the Morrey space $M^p(\Omega)$, $1\leq p \leq \infty$, is the set of all integrable functions such that
$$
\|f\|_{M^p(\Omega)}: =\sup R^{-n(1-1/p)} \int_{\Omega\cap B_R} |f|\, dx<\infty,
$$
where the supremum is taken over all balls $B_R$ of radius $R$. It is well known that for $f\in M^n(\Omega)$ the Riesz potential 
$$
I_{\Omega} f(x)=\int_\Omega |x-y|^{1-n}|f(y)|\, dy
$$
is exponentially summable and satisfies \cite[proof of Lemma 7.20]{GT}
\begin{equation}\label{MP1}
\int_\Omega |I_\Omega f|^q\, dx \leq n(n-1)^{q-1}\omega_n q^q (\mathrm{diam}\,\Omega)^n \|f\|_{M^n(\Omega)}, 
\end{equation}
where positive constants $c_1$ and $c_2$ depend only on $n$, $p$. We shall also use the following simple potential estimate \cite[Lemma 7.12]{GT}. Let $1\leq q\leq \infty$, $0\leq \delta=p^{-1}-q^{-1}<\mu$, $\mu=n^{-1}$. Then
\begin{equation}\label{TI1}
\|I_\Omega f\|_{L^q(\Omega)} \leq \left(\frac{1-\delta}{\mu-\delta} \right)^{1-\delta}\omega_n^{1-\mu} |\Omega|^{\mu-\delta}\|f\|_{L^p(\Omega)}.
\end{equation}

\begin{theorem}\label{T3}
Let $a\in M^n(\Omega)$ and satisfy \eqref{AC}. Then \eqref{Eq2} has a unique solution. The same conclusion also holds if
\begin{equation}\label{C2}
\lim_{\varepsilon\to 0} \varepsilon^{1/n} \|a\|_{L^{n-\varepsilon} (\Omega)}<\infty.
\end{equation}
\end{theorem}
\begin{proof}
 Here we use the same notation as in the proof of Theorem~\ref{T1}. Let $u$ be a solution to \eqref{Eq2} with $f=0$ and $u_\lambda$ be the Lipschitz truncation of $u$. Using $u_\lambda$ as a test function in \eqref{IIvec}, we obtain
 \begin{equation}\label{D1}
 \begin{gathered}
 \int_{\{g\leq \lambda\}} |\nabla u|^2\, dx=-\int_{\Omega} au\nabla u_\lambda\, dx-\int_{\{g>\lambda\}} \nabla u\nabla u_\lambda\, dx\\
 =\int_{\{g>\lambda\}} (a(u_\lambda-u)-\nabla u)\nabla u_\lambda\, dx\leq C\lambda \int_{\{g>\lambda\}}( |a|\cdot |u-u_\lambda|+|\nabla u|)\, dx.
 \end{gathered}
 \end{equation}
For almost all $x\in \Omega$ by the Poincar\'e inequality \cite[Lemma 7.16]{GT} there holds
\begin{equation}\label{D2}
|u-u_\lambda|(x)\leq \frac{(\mathrm{diam}\, \Omega)^n}{|\{g\leq \lambda\}\cap \Omega|}\int_{\{g>\lambda\}}\frac{|\nabla (u-u_\lambda)|}{|x-y|^{n-1}}\, dy.
\end{equation}
Let $\lambda_0$ be such that $|\{g\leq\lambda_0\}\cap \Omega|>\delta (\mathrm{diam}\, \Omega)^n$, $\delta$ a small positive number. Let $V(x)= \int_\Omega |x-y|^{1-n}|a(x)|\, dx$. From \eqref{D1}, \eqref{D2}, for $\lambda\geq \lambda_0$ we have
\begin{multline*}
\int_{\{g\leq \lambda\}} |\nabla u|^2\,dx\leq  C\lambda\int_{\{g>\lambda\}}|\nabla u|\, dx+{}\\
{}+ C\lambda \int_{\{x\in\Omega:\, g(x)>\lambda\}} \int_{\{y\in\Omega:\, g(y)>\lambda\}} |a(x)|\cdot |x-y|^{1-n}\cdot (|\nabla u(y)|+\lambda)  \, dx\, dy\\
\leq C\lambda\int_{\{g>\lambda\}}|\nabla u|\, dx+\int_{\{g>\lambda\}\cap\Omega} (|\nabla u|+\lambda) V\, dx.
\end{multline*}
Multiply this relation by $\varepsilon \lambda^{-1-\varepsilon}$, $0<\varepsilon<1/2$, and integrate with respect to $\lambda$ from $\lambda_0$ to $+\infty$. Fubini's theorem yields
$$
\int_\Omega |\nabla u|^2 (\max (\lambda_0,g))^{-\varepsilon}\, dx \leq
C \varepsilon \int_\Omega (|\nabla u| g^{1-\varepsilon}+ g^{2-\varepsilon}) V\, dx+C\varepsilon \int_\Omega g^{1-\varepsilon}|\nabla u|\, dx.
$$
By H\"older's inequality,
\begin{equation}\label{pf1}
\int_\Omega |\nabla u|^2 (\lambda_0+g)^{-\varepsilon}\, dx \leq C \varepsilon \left[ \left(\int_\Omega g^2\, dx\right)^{\frac{2-\varepsilon}{2}} \left(\int_\Omega V^{2/\varepsilon}\, dx \right)^{\varepsilon/2} + \int_\Omega g^{1-\varepsilon}|\nabla u|\, dx\right].
\end{equation}
If $a\in M^n(\Omega)$, from \eqref{MP1} we obtain
$$
\left(\int_\Omega V(x)^{2/\varepsilon}\, dx \right)^{\varepsilon/2}\leq C (\mathrm{diam}\, \Omega)^{n\varepsilon/2} \frac{\|a\|_{M^n(\Omega)}}{\varepsilon/2}. 
$$
If $a$ satisfies \eqref{C2}, then using \eqref{TI1} with $p=2n/(2+\varepsilon)$ and $q=2/\varepsilon$  we have
$$
\left(\int_\Omega V(x)^{2/\varepsilon}\, dx \right)^{\varepsilon/2} \leq C \varepsilon^{(1-n)/n} |\Omega|^{\varepsilon(n-1)/2n} \|a\|_{L^{n-\varepsilon}(\Omega)}.
$$
Substituting this estimate into \eqref{pf1}, using the boundedness of the Hardy-Littlewood maximal function in $L^2$ and sending $\varepsilon\to 0$, we arrive at
$$
\int_\Omega |\nabla u|^2\, dx \leq C J \int_\Omega g^2 \, dx\leq C J \int_\Omega |\nabla u|^2\, dx
$$
where $J$ is either $ \|a\|_{M^n(\Omega)}$ or $\lim_{\varepsilon\to 0} \varepsilon^{1/2} \|a\|_{L^{n-\varepsilon}(\Omega)}$. Hence, $ \int_\Omega |\nabla u|^2\, dx=0$ provided that $J$ is sufficiently small. This assumption can be removed by the same argument as in the proof of Theorem~\ref{T1}.
\end{proof}
Theorem~\ref{T3} can be obtained as a corollary of Theorem~\ref{T1} if we find a suitable representation of a solution to $\mathrm{div}\, A=a$ in terms of integral potentials with kernels $K(x,y)=O(|x-y|^{1-n})$. In this case applying \eqref{MP1} (or \eqref{TI1}) we would obtain \eqref{ZC1} for $A$. In dimension $2$ this is easy. Also this is simple provided that the normal component of $a$ on $\partial \Omega$ is zero, in which case a solution is given by \eqref{FR1}. In the general case, this is also possible, but requires certain analytical work. Let $n=3$ and $a$ be a smooth solenoidal vector field. Let $\Omega$ be star-shaped with respect to a ball $B\subset \Omega$. For $y\in B$ and $x\in \Omega$ the function $w(x,y)=\int_0^1 \nabla a(y+t(x-y)) \times t (x-y)\, dt$ satisfies $\mathrm{rot}_x\, w=a$. Let $\psi \in C_0^\infty(B)$, $\int_\Omega \psi\, dx=1$. The function $W(x)=\int_B \psi(y) w(x,y)\, dy$ solves $\mathrm{rot} \,W =a$. Interchanging the order of integration, we arrive at 
$$
W(x)=\int_\Omega a(y) \times\frac{x-y}{|x-y|^3}  \int_{|x-y|}^{+\infty} \left(1-\frac{|x-y|}{r} \right)r^{2} \psi \left(x+r\frac{y-x}{|y-x|}\right)\, dr \, dy.
$$
There remains the task of checking the validity of this formula (say, in the spirit of \cite{L3}), and for domains of more complex geometry this is not directly applicable. In the proof of Theorem~\ref{T3} we circumvent these problems.

Condition \eqref{C2} means that $a$ is from the grand Lebesgue space $L^{n)}(\Omega)$ introduced by Iwaniec and Sbordone \cite{IS1}, which is the set of functions $f$ integrable to any power less than $n$ with the finite norm $\|f\|_{L^{n)}(\Omega)}=\sup_{1\leq s <n} \left((n-s) |\Omega|^{-1} \int_\Omega |f|^{s}\, dx \right)^{1/s}$.
Clearly, this condition is satisfied for $a$  from the Marcinkiewicz weak-$L^{n}(\Omega)$ space, i.e. $|\{|a|>t\}|\leq C t^{-n}$. The Orlicz space $L^n \log^{-1}L(\Omega)$ is also contained in $L^{n)}(\Omega)$, and $L^{n)}(\Omega) \subset L^{n}\log^\alpha L (\Omega)$ for all $\alpha<-1$. Further account of properties of grand Lebesgue spaces and their investigation by methods of interpolation theory can be found in \cite{FK}. The closure of $L^n(\Omega)$ in $L^{n)}(\Omega)$ is strictly less than the latter space and is characterized by $\limsup_{\varepsilon \to 0} \varepsilon^{1/n} \|f\|_{L^{n-\varepsilon}(\Omega)}=0$.


For a solenoidal vector field $a\in L^1(\Omega)$ one can easily construct an approximation solution of \eqref{Eq2} for bounded right-hand sides, $f\in L^\infty(\Omega)$ (or, say, $f=f_{i,x_i}+g$, $g\in L^{q/2}(\Omega)$, $f_i\in L^q(\Omega)$, $q>n$). This fact follows from the supremum estimate $\|u\|_{L^\infty(\Omega)}\leq C \|f\|_{L^\infty(\Omega)}$ which is valid for bounded solenoidal $a$ with the constant $C$ independent of $a$. Applying the same reasoning as in Theorem~\ref{T3}, one can prove the uniqueness of approximation solution of \eqref{Eq2} with $a$ from the Morrey space $M^n(\Omega)$ and the right-hand side $f$ from $L^\infty(\Omega)$ without requiring \eqref{AC}. Now, let $f$ be an arbitrary element of $W^{-1,2}(\Omega)$. It can be approximated by bounded $f_j$. Let $u_j$ be approximation solutions of \eqref{Eq2} corresponding to $f_j$. Since in this case approximation solutions are uniquely defined, the difference of any two approximation solutions is also an approximation solution, satisfying $\|u_j-u_k\|_{W_0^{1,2}(\Omega)}\leq C \|f_j-f_k\|_{W^{-1,2}(\Omega)}$. Therefore, the sequence $u_j$ has a strong limit $u$, which does not depend on the choice of approximation of $f$. It would be natural to call this limit a solution to \eqref{Eq2} corresponding to the right-hand side $f$. The limit function can be unbounded, so the term $au$ need not be integrable here. The question is how to understand the equation. For instance, using the Sobolev representation and Fubini's theorem, for bounded $u$ we can transform  the drift term in the integral identity as follows:
$$
\int_\Omega au\nabla \varphi\, dx=(n\omega_n)^{-1}\int_\Omega \nabla u(y)\, dy \int_\Omega \frac{ (x-y)a\nabla \varphi (x) }{|x-y|^n}\, dx, 
$$
which is well defined ($a\nabla \varphi\in M^n(\Omega)$ if $\nabla\varphi$ is bounded) and allows the passage to the limit with respect to the convergence of $\nabla u$ in $L^2(\Omega)$.

\textbf{Acknowledgements}: The work was partially supported by the Russian Foundation for Basic Research, project \textnumero 15-01-00471, and by the Ministry of Education and Science of the Russian Federation, research project \textnumero 1.3270.2017/4.6.

\end{document}